\documentclass[12pt,14paper,reqno]{amsart}
\usepackage{amsmath,amsfonts,amssymb}

\theoremstyle{plain}
\newtheorem{theorem}{Theorem}
\newtheorem{lemma}{Lemma}

\newtheorem{corollary}{Corollary}[section]
\newtheorem{proposition}{Proposition}[section]
\theoremstyle{definition}

\numberwithin{equation}{section}
\numberwithin{lemma}{section}
\numberwithin{theorem}{section}
\usepackage{amsmath}
\usepackage{amsfonts}   
\usepackage{amssymb}
\usepackage{amssymb, amsmath, amsthm}
\usepackage[breaklinks]{hyperref}
\theoremstyle{thm}

\usepackage{graphicx}
\begin{document}
\begin{abstract} 
Let $d$ be a square-free positive integer and $h(d)$ the 
class number of the real quadratic field $\mathbb{Q}
{(\sqrt{d})}.$ In this paper we give an explicit lower bound for 
$h(n^2+r)$, where $r=1,4$, and also establish an equivalent criteria  
to attain this lower bound in terms of special value of Dedekind 
zeta function. Our bounds enable us to reduce the real quadratic families considered in Chowla and Yokoi's conjecture to comparatively small subfamily. We also give an equivalent criteria for having an alternate proof of both the conjectures. Also applying our results, we obtain some criteria for class group of prime power order to be cyclic.
\end{abstract}

\title[Lower bound for class number]{Lower bound for class 
number of certain real quadratic fields}
\author{Mohit Mishra}
\address{Mohit Mishra @Mohit Mishra, Harish-Chandra Research 
Institute, HBNI,
Chhatnag Road, Jhunsi,  Allahabad 211 019, India.}
\email{mohitmishra@hri.res.in}

\maketitle

\section{\textbf{Introduction}}
It is interesting to find bounds 
for class number of a number field. H. Hasse 
\cite{Has65} 
and  H. Yokoi \cite{Yok68, Yok70}  studied lower bounds for  
class numbers of certain real quadratic fields.  Mollin 
\cite{Mol86,Mol87} generalized their results for certain real 
quadratic and biquadratic fields. The author along with 
Chakraborty and Hoque \cite{CHM19} derived a lower bound for 
$\mathbb{Q}(\sqrt{n^2+r})$, where 
$r=1,4$, to classify the class group of order $4$. The bound 
obtained was not so effective.\\
  Gauss conjectured that there exist 
infinitely many real quadratic fields of class number $1$, which 
is yet to be proved. More precisely, he conjectured that there 
exist infinitely many real quadratic fields of the form $
\mathbb{Q}(\sqrt{p}),~ p\equiv 1\pmod4$ of class number $1$. 
Many fruitful research have been done in this direction. 
In this connection, the following two
conjectures were given by Chowla \cite{CF76} and Yokoi \cite{YO86}.
\begin{itemize}
\item[(C)]\label{C1} If $d=m^2+1$ is a prime with $m>26$, then 
$h(d)>1$.

\item[(Y)]\label{C3} Let $d=m^2+4$ be a square-free integer for 
some positive integer $m$. Then there exist exactly $6$ real 
quadratic fields $\mathbb{Q}(\sqrt{d})$ of class number one, 
viz.
$m\in\{1,3,5,7,13,17\}$.
\end{itemize}

Mollin \cite{Mol872} considered $m^2+1$ to be square-free and proved that:

\begin{theorem}\label{thm1.1}
If $m\neq 1,2q$, where $q$ is a prime, then $h(m^2+1)>1$.
\end{theorem}
 He also deduced that if $h(m^2+1)=1$, then $m^2+1$ is a prime. Assuming generalized Riemann hypothesis, Mollin and Williams \cite{MW88}  proved (C) in 1988.
Kim, Leu and Ono \cite{KLO87} proved that at least one of (C) 
and (Y) is true, and for the other case at most $7$ real 
quadratic fields $\mathbb{Q}(\sqrt{d})$ of class number $1$ are 
there. Finally, Bir\'{o} in \cite{ B03, BI03} proved (C) and (Y). For a given fixed number $h$,
it is also interesting to find necessary and 
sufficient conditions for a real quadratic field to have class 
number $h$. Yokoi \cite{YO86} established such kind of criteria 
for $\mathbb{Q}(\sqrt{4m^2+1})$, where $m$ is a positive integer, and he proved the following result:

\begin{theorem}\label{thm1.2}
$h(4m^2+1)=1$ if and only if $m^2-t(t+1)$ is a prime, where 
$1\leq t\leq m-1$.
\end{theorem} 

 Hoque and Saikia \cite{HS15} showed that there are no 
real quadratic fields of the form $\mathbb{Q}(\sqrt{9(8n^2 + r) + 2})$, where $n\geq 1$ and $r = 5, 7$, with class number $1$. Analogously, Chakraborty and Hoque \cite{CH-18} proved that none of the real quadratic fields $\mathbb{Q}(\sqrt{n^2p^2 + 1})$ have class number $1$ when $p\equiv \pm 1 \pmod 8$ is a prime and $n$ is an odd integer. In \cite{CH-19}, they proved that if $d$ is a square-free part of $an^2 + 2$, where $a = 9, 196$ and $n$ is an odd integer, then the class number of
$\mathbb{Q}(\sqrt{d})$ is greater than $1$. Byeon and Kim \cite{BK1,BK2} 
obtained an equivalent criteria 
for R-D type real quadratic fields to have class number $1$ and 
$2$. In \cite{CHM18}, the author along with Chakraborty 
and Hoque  obtained analogous criteria for class 
number $3$.

%\begin{proposition}\label{prop1}Let $d=n^2+1$ be a square free integer. Consider the real quadratic field $k_n=\mathbb{Q}(\sqrt{d})$ and 
%\begin{align*}
%\mathfrak{S}_n:&=\{p\mid n: p \text{ is an odd prime number}\},$\\
%\mathcal{N}:& =\# \mathfrak{S}_n.
%$\end{align*}
%If $\mathcal{N}\geq 3$ then 
%$$h(d) \geq \begin{cases} \mathcal{N}+2 & \text{ if }  d \equiv 1,2,6 \pmod 8, \\
%\mathcal{N}+1 & \text{ if }  d \equiv 5 \pmod 8. \\
%\end{cases}$$
%Moreover, equality holds if and only if 
%$$\zeta_{k_n}(-1)=\begin{cases} \frac{n^3+32n}{288} + \sum\limits_{p_i\in \mathfrak{S}_n} \frac{n^3+n(4p_{i}^4+10p_{i}^2)}{360p_{i}^2} & \text{ if }  d \equiv 1 \pmod 8, \\
%\frac{n^3+5n}{36} + \sum\limits_{p_i\in \mathfrak{S}_n} \frac{4n^3+n(p_{i}^4+10p_{i}^2)}{180p_{i}^2} & \text{ if }  d \equiv 2,6 \pmod 8, \\
%\frac{n^3+14n}{360} + \sum\limits_{p_i\in \mathfrak{S}_n} \frac{n^3+n(4p_{i}^4+10p_{i}^2)}{360p_{i}^2} & \text{ if }  d \equiv 5 \pmod 8.\\
%\end{cases}$$ 
%\end{proposition}

  In this paper, we establish an efficient lower bound for the
class number of $\mathbb{Q}(\sqrt{n^2+r})$, 
where $r=1,4$. Applying these results on Chowla's
conjecture(C), we gave an alternate proof of Theorem \eqref{thm1.1}. We also extended Theorem \eqref{thm1.2} and proved that:

\begin{theorem}\label{thm1.3}
 If $m^2+1$ is a square-free integer, then $h(m^2+1)=1$ if and only if $m=1 \text { or } 2p$, for some prime $p$, and $p^2-t(t+1)$ is a prime, 
$\forall \hspace*{2mm} 1\leq t\leq p-1$.
\end{theorem}

Thus Chowla's conjecture is equivalent to proving that \textit{the set $\{p^2-n(n+1)\}_{n=1}^{p-2}$ always contains a composite number, except for $p = 2,3,5,7 \text { and } 13$}.
\\

We also obtained similar kind of results for Yokoi's conjecture and proved the following results;

\begin{theorem}\label{thm1.4}
Let $m^2+4$ is a square-free integer and if $h(m^2+4)=1$, then  $m=1 \text{ or } p$, for some prime $p$.
\end{theorem}

\begin{theorem}\label{thm1.5}
If $m^2+4$ is a square-free integer, then $h(m^2+4)=1$ if and only if $m=1 \text{ or } p$, for some prime $p$, and $1+pt-t^2$ is a prime,  $\forall \hspace{2mm} 1 \leq t \leq \frac{p-1}{2}$.
\end{theorem} 
 
Thus Yokoi's conjecture is equivalent to proving that \textit{  $\{1+pt-t^2\}_{t=1}^{\frac{p-1}{2}}$ always contains a 
composite number, except for $p =3,5,7,
13 \text { and } 17$}. \\
 
Chowla and Friedlander \cite{CF76} proved that if $m^2+1$ is a prime with $m>2$ and $h(m^2+1)=1$, then $g(m^2+1)$ is $\frac{m}{2}$, where $g(n)$ is the least prime 
which is a quadratic residue of $n$.
We generalize this result and give an upper bound on 
$g(1+4p^2)$, where $p$ is a prime and $h(1+4p^2)>1$, and for $g(p^2+4)$, where $p$ is a prime and $h(p^2+4)>1$. 
We also give some criteria for prime power order class group of a
real quadratic field $\mathbb{Q}(\sqrt{n^2+r})$, where $r=1,4$, 
to be cyclic.

\subsection*{Notations and structure of paper} Throughout this 
paper $d=n^2+r$ be a square-free integer. Let $h(d)$ and $
\mathfrak{C}(k_n)$ denote the class number and the class group 
of a real quadratic field $k_n=\mathbb{Q}(\sqrt{d})$, 
respectively. $\mathcal{P}$ will 
always denote the principal ideal class in the class group and $N(\mathcal{I})$ denote the norm of an ideal $\mathcal{I}$.  $p^t \mid \mid n$ means that $p^t \mid n$ but $p^{t+1} \nmid n$.
In $\S 2$  we stated the results on computing the partial 
Dedekind zeta values of a real quadratic field. In $\S3$ 
and $\S4$ we compute the partial Dedekind zeta values and with 
some group theoretic arguments we deduce a lower bound for 
class number of  $\mathbb{Q}(\sqrt{n^2+r})$ for $r=1,4$. In $
\S5$ we apply our results on Chowla and Yokoi's 
conjecture and deduce some results. In $\S6$ the class groups of 
prime power order is being studied. In $\S7$ we deduce some 
corollaries of our main results of $\S3$ and $\S4$. We conclude  this paper with some remarks.

\section{\textbf{Partial Dedekind zeta values}}
Let $k$ be a real quadratic field, and $\zeta_{k}(s)$ be the 
Dedekind zeta function attached to $k$. Siegel \cite{SI69} derived an expression for the Dedekind zeta 
values at $1-2n$, where $n$ is a positive integer. For $n=1$, 
this expression becomes simpler:

\begin{proposition}\label{prop2.1}
Let $D$ be the discriminant of $k$. Then 
$$\zeta_k(-1)=\frac{1}{60}\sum_{\substack{ |t|<\sqrt{D}\\ 
t^2\equiv D\pmod 4}}\sigma\left(\frac{D-t^2}{4}\right),$$
where $\sigma(n)$ denotes the sum of divisors of $n$.\
\end{proposition}

Lang gave another method to compute $\zeta_k(-1)$ by computing partial 
Dedekind zeta values and summing them up. For an ideal class $
\mathfrak{A}\subset k$, consider an integral ideal $\mathfrak{a}
$ in $\mathfrak{A}^{-1}$ with integral basis $\{r_{1},r_{2}\}$. Let $r_1'$ and  $r_2'$ be the 
conjugates of $r_1$ and $r_2$, respectively, and $$
\delta(\mathfrak{a}):= r_1r_2'-r_1'r_2.$$

Let $\varepsilon$ be the fundamental unit of $k$. Then  one has  a matrix
$M=
\begin{bmatrix}
a&b\\
c&d
\end{bmatrix}
$
with integer entries satisfying: 
$$\varepsilon\begin{bmatrix}
r_1\\r_2
\end{bmatrix}
=M\begin{bmatrix}
r_1\\
r_2
\end{bmatrix}.$$
We recall Lang's result \cite{LAN}.

\begin{theorem}\label{thm2.1}
By keeping the above notations, we have 
\begin{align*}
\zeta_k(-1, \mathfrak{A})&=\frac{\textsl{sgn }
\delta(\mathfrak{a})~r_2r_2'}{360N(\mathfrak{a})c^3}\big\{(a
+d)^3-6(a+d)N(\varepsilon)-240c^3(\textsl{sgn } c)\\
&\times S^3(a,c)+180ac^3(\textsl{sgn } 
c)S^2(a,c)-240c^3(\textsl{sgn } c)S^3(d,c)\\
& +180dc^3(\textsl{sgn } c)S^2(d,c) \big\},
\end{align*}
where $S^i(-,-)$ denotes the generalized Dedekind sum as defined 
in \cite{AP50}.
\end{theorem} 

In order to apply Theorem \eqref{thm2.1} one needs to know $a,b,c,d$ and the generalized Dedekind sums. The following 
result \cite[p. 143, Eq. 2.15]{LAN} helps us to determine $a, b, c$ and $d$.
\begin{lemma}\label{lem2.1} With the same notations as above, we have
$$
M=
\begin{bmatrix}
Tr\left(\frac{r_1 r_2'\varepsilon}{\delta(\mathfrak{a})}\right)& 
Tr\left(\frac{r_1 r_1'\varepsilon'}{\delta(\mathfrak{a})}\right) 
\vspace*{2mm} \\ 
Tr\left(\frac{r_2 r_2'\varepsilon}{\delta(\mathfrak{a})}\right) 
&
Tr\left(\frac{r_1 r_2'\varepsilon'}{\delta(\mathfrak{a})}
\right)
\end{bmatrix}
$$
Moreover, $bc\ne 0$ and $\det(M)=N(\varepsilon)$.
\end{lemma} 
To compute partial zeta values for ideal classes of a real 
quadratic field the following expressions (see, \cite[pp. 
155, Eq. 4.3--4.4]{LAN}) are required.
\begin{lemma}\label{lem2.2} For any positive integer $m$, we have
\begin{itemize}
\item[(i)] $S^3(\pm 1, m)=\pm\frac{-m^4+5m^2-4}{120m^3}.$
\vspace*{2mm}
\item[(ii)] $S^2(\pm 1, m)=\frac{m^4+10m^2-6}{180m^3}.$
\end{itemize}
\end{lemma}
%\begin{lemma}\label{DS2} For any positive even integer $m$, we 
%have
%\begin{itemize}
%\item[(i)] $S^3(m\pm 1, 2m)=\pm S^1(m+1, 2m)=\mp
%\frac{m^4-50m^2+4}{960m^3}.$ \vspace*{2mm}
%\item[(ii)] $S^2(m-1, 2m)=S^2(m+1, 2m)=\frac{m^4+100m^2-6}
%{1440m^3}.$\vspace*{2mm}
%\item[(iii)] $S^3(m+1, 4m)=\frac{-m^4-180m^3+410m^2-4}{7680m^3}.
%$\vspace*{2mm}
%\item[(iv)] $S^3(m-1, 4m)=\frac{m^4-180m^3-410m^2+4}{7680m^3}.$
%\vspace*{2mm}
%\item[(v)] $S^2(m-1, 4m)=S^2(m+1, 4m)=\frac{m^4+820m^2-6}
%{11520m^3}.$
%\end{itemize}
%\end{lemma}

\section{\textbf{The field $\mathbf{\mathbb{Q}(\sqrt{n^2+1})}$}}

In this case, fundamental unit is $\varepsilon= n+\sqrt{n^2+1} \text { and } N(\varepsilon)=-1$. We also know that if $p|n$, where $p$ is an odd prime, then $p$ splits in 
$k_n$ as
\begin{equation}\label{eqn3.1}
(p)=\begin{cases}
\left(p,\frac{1+\sqrt{d}}{2}\right)\left(p,\frac{1-\sqrt{d}}
{2}\right)  &{\rm ~if~} n^2+1 \equiv1 \pmod 4,\\
(p,1+\sqrt{d})(p,1-\sqrt{d})  &{\rm ~if~ } n^2+1 \equiv2 \pmod 4.
\end{cases}
\end{equation}

By \cite[Theorem 2.3]{BK1}, we have  

\begin{equation}\label{eqn3.2}
\zeta_{k_n}(-1, \mathcal{P})=\begin{cases}
\frac{n^3+14n}{360}  &{\rm ~if~} n^2+1 \equiv1 \pmod 4,\\
\frac{4n^3+11n}{180}  &{\rm ~if~} n^2+1 \equiv2 \pmod 4.
\end{cases}
\end{equation}

We first find the integral basis of some particular ideals. Using this and Theorem \ref{thm2.1}, we calculate the partial Dedekind zeta values for an ideal 
classes of $k_n$. Then we compare 
these values with Dedekind zeta value and use some elementary 
group theoretic arguments to establish our results.

\begin{lemma}\label{lem3.1}
Let $n^2+1 \equiv 1 \pmod 4$, $p^t \mid\mid n$ be a prime, $\mathfrak{a}=\left(p,\frac{1+\sqrt{d}}{2}\right)$ and $\mathfrak{a}^{-1}=\left(p,\frac{1-\sqrt{d}}{2}\right)$. Then $\{p^r,\frac{1+\sqrt{d}}{2}\}$ and $\{p^r,\frac{1-\sqrt{d}}{2}\}$ is an 
integral basis for $\mathfrak{a}^r$ and $\mathfrak{a}^{-r}$, respectively, $ \forall \hspace*{2mm} 1 \leq r \leq t$. 
\end{lemma}

\begin{proof}
Consider $$M_{r}=\Big[p^r,
\frac{1+\sqrt{d}}{2}\Big],$$ a nonzero $
\mathbb{Z}$-module in $\mathcal{O}_{k_n}$.
Then, by \cite[Proposition 2.6 and 2.11]{Fra}, $M_{r}$ is an ideal and $N(M_{r})=p^r$, $\forall \hspace{2mm} 1 \leq r \leq t$. As $N(\mathfrak{a}^r)=p^r$ and $ \mathfrak{a}^r \subseteq M_{r}$, $\forall \hspace{2mm} 1 \leq r \leq t$, therefore $M_{r} = 
\mathfrak{a}^r$. Hence $\{p^r,\frac{1+\sqrt{d}}{2}\}$  is an 
integral basis for $\mathfrak{a}^r$, $\forall \hspace{2mm} 1 \leq r \leq t$. Similarly, if $M_{r}'=
\Big[p^r,\frac{1-\sqrt{d}}{2}\Big]$ then $M_{r}' = \mathfrak{a}^{-r}
$ and $\{p^r,\frac{1-\sqrt{d}}{2}\}$  is an integral basis 
for $\mathfrak{a}^{-r}$, $\forall \hspace{2mm} 1 \leq r \leq t$.
\end{proof}

Now if we consider non-zero $\mathbb{Z}$-modules $N_r=[p^r,1+\sqrt{d}]$ and $N_{r}'=[p^r,1-\sqrt{d}]$ in $\mathcal{O}_{k_n}$, where $n^2+1 \equiv 2 \pmod 4$, then as before one can prove the following;

\begin{lemma}\label{lem3.2}
Let $n^2+1 \equiv 2 \pmod 4$, $p^t \mid\mid n$ be a prime, $\mathfrak{a}=(p,1+\sqrt{d})$ and $\mathfrak{a}^{-1}=(p,1-\sqrt{d})$. Then $\{p^r,1+\sqrt{d}\}$  and $\{p^r,1-\sqrt{d}\}$ is an 
integral basis for $\mathfrak{a}^r$ and $\mathfrak{a}^{-r}$, respectively, $ \forall \hspace*{2mm} 1 \leq r \leq t$. 
\end{lemma}

 We will derive our results in three subsections based on the congruence 
relation, i.e. $n^2+1 \equiv 1,2,5 \pmod 8$.

\subsection{$\mathbf{n^2+1\equiv 5 \pmod 8}$} 

$n^2 \equiv 4 \pmod 8 \Rightarrow n=2n_0$, where 
$n_0$ is an odd integer.

%\noindent{Case I:-} Let $n_0=p^t$ for some integer $t$. In this 
%case $p$ splits in $k_n$ as

\begin{theorem}\label{thm3.1}
Let $p$ be an odd prime and $n=2p^t$ with $t \geq 1$ an integer. 
Then $h(d) \geq t$ and equality holds if and only if  $$
\zeta_{k_n}(-1)=\frac{n^3+14n}{360}+\sum_{r=1}^{t-1} 
\frac{n^3+n(4p^{4r}+10p^{2r})}{360p^{2r}}.$$
\end{theorem}
\begin{proof}
By \eqref{eqn3.1} $p$ splits in $k_n$, so let $\mathcal{A}$ be an ideal class containing $\mathfrak{a}=
\left(p,\frac{1+\sqrt{d}}{2}\right)$. Then  $\mathfrak{a}^{-1}=
\left(p,\frac{1-\sqrt{d}}{2}\right) \in \mathcal{A}^{-1}$ and, by Lemma \ref{lem3.1}, $\{p^r,\frac{1+\sqrt{d}}{2}\}$ and $\{p^r,\frac{1-\sqrt{d}}{2}\}$ is an 
integral basis for $\mathfrak{a}^r$ and $\mathfrak{a}^{-r}$, $\forall \hspace*{2mm} 1 \leq r \leq t$.
Now by using Lemma \ref{lem2.1}, Lemma \ref{lem2.2} and Theorem 
\ref{thm2.1}, we get
$$\zeta_{k_n}(-1, \mathcal{A}^{r})=\frac{n^3+n(4p^{4r}
+10p^{2r})}{360p^{2r}}=\zeta_{k_n}(-1, \mathcal{A}^{-r}), \hspace{10mm} \forall \hspace{2mm} 1\leq 
r \leq t.$$
 If for any $1\leq r \leq t$, $\mathcal{A}^r=\mathcal{P}$ then $
\zeta_{k_n}(-1, \mathcal{A}^r)=\zeta_{k_n}(-1, \mathcal{P})$ 
which gives $n=2p^r$. Therefore, $
\mathcal{A}^r$ is non-principal ideal class $\forall \hspace*{2mm} 1 \leq r < t$ and this implies $|
\mathcal{A}| \geq t.$ Hence $h(d) \geq t$ and
$$\zeta_{k_n}(-1) \geq \zeta_{k_n}(-1, \mathcal{P})+\sum_{r=1}
^{t-1}\zeta_{k_n}(-1, \mathcal{A}^r).$$
Equality holds if and only if $h(d)=t.$
Thus $h(d)=t$ if and only if
$$\zeta_{k_n}(-1)=\frac{n^3+14n}{360}+\sum_{r=1}^{t-1} 
\frac{n^3+n(4p^{4r}+10p^{2r})}{360p^{2r}}.$$

\end{proof}

\begin{theorem}\label{thm3.2}
Let $n=2{p_{1}}^{a_{1}}{p_{2}}^{a_{2}}\cdots{p_{m}}^{a_{m}}$ 
with $p_{i}$'s are distinct odd primes and $a_i$'s are postive integers.
\begin{itemize}
\item[(i)] If $m>2$, then $h(d) \geq 2(a_{1}+a_{2}+\cdots+ a_{m})-m+1$ 
and equality holds if and only if 
\begin{eqnarray*}
\zeta_{k_n}(-1)
&=&\frac{n^3+14n}{360}+\sum_{\substack{1 \leq i \leq m\\ 1 \leq r_{i} \leq a_{i}-1}} \frac{n^3+n(4{p_{i}}^{4r_{i}}+10{p_{i}}^{2r_{i}})}{180{p_{i}}^{2r_{i}}}\\
& + & \sum_{i=1}^{m} \frac{n^3+n(4{p_{i}}^{4a_{i}}+10{p_{i}}^{2a_{i}})}{360{p_{i}}^{2a_{i}}}.
\end{eqnarray*}
\item[(ii)] If $m=2$, then $h(d) \geq 2(a_{1}+a_{2})-2$ and 
equality holds if and only if 
\begin{eqnarray*}
\zeta_{k_n}(-1)&=&\frac{n^3+14n}{360}+\sum_{\substack{i=1,2 \\ 1 
\leq r_{i} \leq a_{i}-1}} \frac{n^3+n(4{p_{i}}^{4r_{i}}
+10{p_{i}}^{2r_{i}})}{180{p_{i}}^{2r_{i}}}\\
&+& \frac{n^3+n(4{p_{1}}^{4a_{1}}+10{p_{1}}^{2a_{1}})}{360{p_{1}}^{2a_{1}}}.
\end{eqnarray*}

\end{itemize}
\end{theorem}

\begin{proof}
We will prove part (i) and the other part can be proved on the same line.\\
\hspace*{5mm} Since each $p_{i}$ splits in $k_n$, as in \eqref{eqn3.1}, 
so let $\mathcal{A}_i$ be the ideal class containing $\mathfrak{a}
_i=\left(p_{i},\frac{1+\sqrt{d}}{2}\right)$.
Then again by using Lemma \ref{lem2.1}, Lemma \ref{lem2.2} and Theorem \ref{thm2.1}, 
$$\zeta_{k_n}(-1, \mathcal{A}_i^{r_{i}})=\frac{n^3+n(4p_i^{4r_i}
+10p_i^{2r_i})}{360p_i^{2r_i}}=\zeta_{k_n}(-1, \mathcal{A}_i^{-r_{i}}),$$ $\forall 
\hspace{3mm} 1\leq i \leq m\hspace{2mm} \mbox{and} \hspace{2mm} 
1 \leq r_i \leq a_i$.
Compairing the values of $\zeta_{k_n}(-1, \mathcal{A}_i^{r_{i}})
$ and $\zeta_{k_n}(-1, \mathcal{P})$, we get that $\mathcal{A}
_i^{r_i}$ are distinct non-principal ideal classes, $\forall \hspace*{2mm} 1\leq 
i \leq m$ and $1\leq r_i \leq a_i.$
And if for any  $1\leq i \leq m$, $1\leq r_i 
, s_i \leq a_i\text { and } r_i \neq s_i$ we get $\zeta_{k_n}(-1, \mathcal{A}_i^{r_{i}})=\zeta_{k_n}(-1, \mathcal{A}_i^{-s_{i}})$, then we have $n=2p^{r_i+s_i}$, which is not 
possible. This implies $\mathcal{A}_i^{r_i} \neq \mathcal{A}_i^{-s_i}, \hspace*{2mm} \forall \hspace*{2mm} 1\leq i\leq m, \hspace*{1mm}  1 \leq r_i,s_i \leq a_i \text { and } r_i \neq s_i$. Therefore, $|\mathcal{A}_i| \geq 2{a}_i$ and hence 
$h(d) 
\geq 2(a_{1}+a_{2}+\cdots+ a_{m})-m+1$.
Now, $h(d) = 2(a_{1}+a_{2}+\cdots+ a_{m})-m+1$(only if $m$ is 
odd) if and only if 
$$\zeta_{k_n}(-1)=\zeta_{k_n}(-1, \mathcal{P})+\sum_{\substack{1 
\leq i \leq m\\ 1 \leq r_{i} \leq a_{i}-1}} 2\zeta_{k_n}(-1, 
\mathcal{A}_i^{r_i})+\sum_{i=1}^{m}\zeta_{k_n}(-1, \mathcal{A}
_i^{a_i}).$$
Thus $h(d) = 2(a_{1}+a_{2}+\cdots+ a_{m})-m+1$,  if and only if 
\begin{eqnarray*}
\zeta_{k_n}(-1)
&=&\frac{n^3+14n}{360}+\sum_{\substack{1 \leq i \leq m\\ 1 \leq 
r_{i} \leq a_{i}-1}} \frac{n^3+n(4{p_{i}}^{4r_{i}}+10{p_{i}}
^{2r_{i}})}{180{p_{i}}^{2r_{i}}}\\
& + & \sum_{i=1}^{m} \frac{n^3+n(4{p_{i}}^{4a_{i}}+10{p_{i}}
^{2a_{i}})}{360{p_{i}}^{2a_{i}}}.
\end{eqnarray*}
\end{proof}

\noindent{\textit{Remark:}} If all the $a_i$'s and $t$ are zero, then $n_0=1$. Hence $d=5$ and $h(5)=1$.

\subsection{$\mathbf{n^2+1\equiv 1 \pmod 8}$}  

In this case, $4|n$ and $2$ splits in  $k_n$ as
\begin{equation}\label{eqn3.3}
(2)=\left(2,\frac{1+\sqrt{d}}{2}\right)\left(2,\frac{1-\sqrt{d}}{2}\right).
\end{equation}

\begin{theorem}\label{thm3.3}
Let $n=2^sp^t$ with $s>1$, $t \geq 1$ and $p$ be an odd prime. Then $h(d) \geq 2(t+s)-4$ and equality holds if and only if 
\begin{eqnarray*}
\zeta_{k_n}(-1) &=& \frac{n^3+14n}{360}+\sum_{r=1}^{t-1} \frac{n^3+n(4p^{4r}+10p^{2r})}{180p^{2r}}\\
&+&\frac{n^3+n(4p^{4t}+10p^{2t})}{360p^{2t}}+\sum_{j=1}^{s-2} \frac{n^3+n(4\times2^{4j}+10\times2^{2j})}{180\times2^{2j}}.
\end{eqnarray*}
\end{theorem}

\begin{proof}
Let $\mathcal{A}$ and $\mathcal{B}$ be the two ideal classes in $k_n$ such that $\mathfrak{a}=\left(p,\frac{1+\sqrt{d}}{2}\right) \in \mathcal{A}$ and $\mathfrak{b}=\left(2,\frac{1+\sqrt{d}}{2}\right) \in \mathcal{B}.$ Then as before,$$\zeta_{k_n}(-1,\mathcal{A}^r)=\frac{n^3+n(4p^{4r}+10p^{2r})}{360p^{2r}}=\zeta_{k_n}(-1,\mathcal{A}^{-r}), \hspace*{7mm} \forall \hspace*{2mm} 1 \leq r \leq t,$$ 
and 
\begin{eqnarray*}
\zeta_{k_n}(-1,\mathcal{B}^j)&=&\frac{n^3+n(4\times 2^{4j}+10 \times 2^{2j})}{360 \times 2^{2j}}\\
&=&\zeta_{k_n}(-1,\mathcal{B}^{-j}), \hspace*{40mm} \forall \hspace*{3mm}1 \leq j \leq s-1.
\end{eqnarray*}

If $\zeta_{k_n}(-1, \mathcal{P})=\zeta_{k_n}(-1, \mathcal{A}^r)$, then $n=2p^r$. This shows that $\mathcal{A}^r$ is a non-principal ideal class, $\forall \hspace*{2mm} 1 \leq r \leq t$. As $\zeta_{k_n}(-1,\mathcal{A}^r)=\zeta_{k_n}(-1,\mathcal{A}^{-r})$ and we have $\zeta_{k_n}(-1,\mathcal{A}^r)\neq \zeta_{k_n}(-1,\mathcal{A}^s)$, $\forall \hspace*{2mm} 1 \leq r,s \leq t, \hspace*{1mm} r \neq s $, hence $\mathcal{A}^{r} \neq \mathcal{A}^{-s}, \hspace*{2mm} \forall \hspace*{2mm} 1 \leq r,s \leq t \text { and } r \neq s$. Thus $|\mathcal{A}| \geq 2t$. Now if $\zeta_{k_n}(-1, \mathcal{P})=\zeta_{k_n}(-1, \mathcal{B}^j)$, we get $n=2\times2^j$, therefore as above, $|\mathcal{B}| \geq 2(s-1)$. And if $\zeta_{k_n}(-1, \mathcal{A}^r)=\zeta_{k_n}(-1, \mathcal{B}^j)$, we have $n=2\times2^jp^r$, this is only  possible when $r=t$ and $j=s-1$. So, while calculating class number we have to take care of $\mathcal{A}^t$ and $\mathcal{B}^{s-1}$ as they may be equal and we have to count them once. 
Hence $h(d) \geq (2t-1)+(2(s-1)-1)+1-1$, i.e. $h(d) \geq 2(t+s)-4$ and equality holds if and only if 
$$\zeta_{k_n}(-1)=\zeta_{k_n}(-1, \mathcal{P})+\sum_{r=1}^{t-1}2\zeta_{k_n}(-1, \mathcal{A}^r)+\zeta_{k_n}(-1, \mathcal{A}^t)+\sum_{j=1}^{s-2}2\zeta_{k_n}(-1, \mathcal{B}^j),$$
$$\text{or}$$
$$\zeta_{k_n}(-1)=\zeta_{k_n}(-1, \mathcal{P})+\sum_{r=1}^{t-1}2\zeta_{k_n}(-1, \mathcal{A}^r)+\zeta_{k_n}(-1, \mathcal{B}^{s-1})+\sum_{j=1}^{s-2}2\zeta_{k_n}(-1, \mathcal{B}^j),$$
since $\zeta_{k_n}(-1, \mathcal{A}^t)=\zeta_{k_n}(-1, \mathcal{B}^{s-1})$.
Thus $h(d)=2(t+s)-4$ if and only if
\begin{eqnarray*}
\zeta_{k_n}(-1) &=& \frac{n^3+14n}{360}+\sum_{r=1}^{t-1} \frac{n^3+n(4p^{4r}+10p^{2r})}{180p^{2r}}\\
&+&\frac{n^3+n(4p^{4t}+10p^{2t})}{360p^{2t}}+\sum_{j=1}^{s-2} \frac{n^3+n(4\times2^{4j}+10\times2^{2j})}{180\times2^{2j}}.
\end{eqnarray*}
\end{proof}

Using similar arguments, one can prove:

\begin{theorem}\label{thm3.4}
Let $n=2^s{p_{1}}^{a_{1}}{p_{2}}^{a_{2}}\cdots{p_{m}}^{a_{m}}$ with $p_{i}$'s be distinct odd primes, $a_i$'s are postive integers and  $s,m \geq 2$. If $m$ is even, then $h(d) \geq 2(a_{1}+a_{2}+\cdots+ a_{m})-m+2s-2$ and equality holds if and only if  
\begin{eqnarray*}
\zeta_{k_n}(-1)&=&\frac{n^3+14n}{360}+\sum_{\substack{1 \leq i \leq m\\ 1 \leq r_{i} \leq a_{i}-1}} \frac{n^3+n(4{p_{i}}^{4r_{i}}+10{p_{i}}^{2r_{i}})}{180{p_{i}}^{2r_{i}}}\\
 &+& \sum_{i=1}^{m} \frac{n^3+n(4{p_{i}}^{4a_{i}}+10{p_{i}}^{2a_{i}})}{360{p_{i}}^{2a_{i}}} + \sum_{j=1}^{s-2} \frac{n^3+n(4\times2^{4j}+10\times2^{2j})}{180\times2^{2j}} \\
 &+& \frac{n^3+n(4\times2^{4(s-1)}+10\times2^{2(s-1)})}{360\times2^{2(s-1)}}.
\end{eqnarray*}
If $m$ is odd, then $h(d) \geq 2(a_{1}+a_{2}+\cdots+ a_{m})-m+2s-1$ and $h(d)=2(a_{1}+a_{2}+\cdots+ a_{m})+2s-1$ if and only if  
\begin{eqnarray*}
\zeta_{k_n}(-1)&=&\frac{n^3+14n}{360}+\sum_{\substack{1 \leq i \leq m\\ 1 \leq r_{i} \leq a_{i}}}\frac{n^3+n(4{p_{i}}^{4r{i}}+10{p_{i}}^{2r_{i}})}{180{p_{i}}^{2r_{i}}}\\
&+& \sum_{j=1}^{s-1} \frac{n^3+n(4\times2^{4j}+10\times2^{2j})}{180\times2^{2j}}.
\end{eqnarray*}
\end{theorem}

\noindent{\textit{Remark:}} When all the $a_i$'s and $t$ are zero, then $n=2^s$ and $d=2^{2s}+1$. By similar arguments one can show that $h(d) \geq s-1$.

%\begin{proof}
%Let $t=a_{1}+a_{2}+\cdots+ a_{m}$. Also each $(p_{i})$ splits in $k_n$ as 
%$$(p_{i})=\left(p_{i},\frac{1+\sqrt{d}}{2}\right)\left(p_{i},\frac{1-\sqrt{d}}%{2}\right).$$
%Let $\mathcal{A}_i$ be the ideal class containing $\mathfrak{a}_i=\left(p_{i},\frac{1+\sqrt{d}}{2}\right)$
%Now again by using Lemma \ref{2.1}, Lemma \ref{DS1} and Theorem \ref{thm2.2} we obtain: 
%$$\zeta_{k_n}(-1, \mathcal{A}_i^{r_{i}})=\frac{n^3+n(4p_i^{4r_i}+10p_i^{2r_i})}{360p_i^{2r_i}} \hspace{2mm} \forall 1\leq r_i \leq a_i and 1\leq i \leq m.$$
%Further if $\mathcal{P}$ is the principal ideal class in $k_n$ then $\zeta_{k_n}(-1, \mathcal{P})$ is given \eqref{eq3}.Copmaring the values of $\zeta_{k_n}(-1, \mathcal{A}_i^{r_{i}})$ and $\zeta_{k_n}(-1, \mathcal{P})$, we get that $\mathcal{A}_i^{r_i}$ is non-principal ideal class for all  $1\leq r_i \leq a_i$ and $1\leq i \leq m.$This implies that $|\mathcal{A}_i^{r_i}| \geq 2{a}_i.$ Hence $h(d) \geq 2(a_{1}+a_{2}+\cdots+ a_{m})-m+1$.Now, $h(d) = 2(a_{1}+a_{2}+\cdots+ a_{m})-m+1$(only if $m$ is odd) if and only if $$\zeta_{k_n}(-1)=\zeta_{k_n}(-1, \mathcal{P})+\sum_{1 \leq i \leq m, r_{i}=1}^{a_{i}-1} 2\zeta_{k_n}(-1, \mathcal{A}_i^{r_i})+\sum_{i=1}^{m}\zeta_{k_n}(-1, \mathcal{A}_i^{a_i}).$$This implies,$$\zeta_{k_n}(-1)=\frac{n^3+14n}{360}+\sum_{1 \leq i \leq m, r_{i}=1}^{a_{i}-1} \frac{n^3+n(4{p_{i}}^{4r{i}}+10{p_{i}}^{2r_{i}})}{180{p_{i}}^{2r_{i}}} + \sum_{i=1}^{m} \frac{n^3+n(4{p_{i}}^{4a_{i}}+10{p_{i}}^{2a_{i}})}{180{p_{i}}^{2a_{i}}}.$$
%\end{proof}

\subsection{$\mathbf{d=n^2+1\equiv 2\pmod 4}$}

In this case, we have $n^2+1\equiv 2\pmod 4$ and $n$ is odd. If $\mathfrak{a}=(p,1+\sqrt{d}) \in \mathcal{A}$ then as before, using Lemma \ref{lem2.1}, Lemma \ref{lem2.1}, Lemma \ref{lem3.2} and Theorem \ref{thm2.1}, we get
 
\begin{equation} \label{eqn3.4}
\zeta_{k_n}(-1, \mathcal{A}^r)=\frac{8n^3+n(2p^{4r}+20p^{2r})}{360p^{2r}}=\zeta_{k_n}(-1, \mathcal{A}^{-r}), \hspace*{5mm} \forall \hspace*{2mm} 1\leq r \leq t.
\end{equation}

\begin{theorem}\label{thm3.5}
Let $n=p^t$ with $t \geq 1$ an integer. Then $h(d)$ is even and $h(d) \geq 2t$. Equality holds if and only if  $$\zeta_{k_n}(-1)=\frac{4n^3+11n}{180}+\sum_{r=1}^{t-1} \frac{8n^3+n(2p^{4r}+20p^{2r})}{180p^{2r}}+\frac{8n^3+n(2p^{4t}+20p^{2t})}{360p^{2t}}.$$
\end{theorem}

\begin{proof}
If $d \equiv 2 \mod 4$ then $2$ ramifies in $k_n=\mathbb{Q}{(\sqrt{d})}$, i.e. $$(2)=(2,d)^2.$$
If $\mathfrak{b}=(2,d)$ is in ideal class $\mathcal{B}$, then by \cite[Theorem 2.3]{BK1} we have 
$$\zeta_{k_n}(-1, \mathcal{B})=\frac{2n^3+28n}{360}.$$
Also if $\zeta_{k_n}(-1, \mathcal{P})=\zeta_{k_n}(-1, \mathcal{B})$, then $d=2$. Therefore $\mathcal{B}$ is non-principal ideal class and $|\mathcal{B}|=2$. Hence $h(d)$ is even.\\
Now if $\mathfrak{a}=(p,1+\sqrt{d}) \in \mathcal{A}$, then by \eqref{eqn3.4} we have 
$$\zeta_{k_n}(-1, \mathcal{A}^r)=\frac{8n^3+n(2p^{4r}+20p^{2r})}{360p^{2r}}=\zeta_{k_n}(-1, \mathcal{A}^{-r}), \hspace*{5mm} \forall \hspace*{3mm} 1\leq r \leq t.$$
Now if $\zeta_{k_n}(-1, \mathcal{P})=\zeta_{k_n}(-1, \mathcal{A}^r)$, then $n=\frac{p^r}{2}$, which is not possible. Hence $\mathcal{A}^r$ is a non-principal ideal class,  $ \forall \hspace*{2mm} 1 \leq r \leq t.$
Also $\zeta_{k_n}(-1, \mathcal{A}^r)=\zeta_{k_n}(-1, \mathcal{A}^{-r})$, $ \forall \hspace*{2mm} 1 \leq r \leq t$, and if $\zeta_{k_n}(-1, \mathcal{A}^r)=\zeta_{k_n}(-1, \mathcal{A}^s)$, where $1 \leq r,s \leq t$ and $r \neq s$, then $n=\frac{p^{r+s}}{2}$, which is again not possible. This gives $|\mathcal{A}| \geq 2t.$ 
Similarly, $\mathcal{B} \neq \mathcal{A}^r$, $\forall \hspace*{2mm} 1 \leq r < t$ and therefore $h(d) \geq 2t$. Equality holds if and only if 
\begin{eqnarray*}
\zeta_{k_n}(-1)&=&\zeta_{k_n}(-1, \mathcal{P})+\sum_{r=1}^{t-1}2\zeta_{k_n}(-1, \mathcal{A}^r)+\zeta_{k_n}(-1, \mathcal{A}^t)\\
&=&\zeta_{k_n}(-1, \mathcal{P})+\sum_{r=1}^{t-1}2\zeta_{k_n}(-1, \mathcal{A}^r)+\zeta_{k_n}(-1, \mathcal{B}),
\end{eqnarray*}
as $\zeta_{k_n}(-1, \mathcal{A}^t)=\zeta_{k_n}(-1, \mathcal{B})$. Thus $h(d)=2t$ if and only if
$$\zeta_{k_n}(-1)=\frac{4n^3+11n}{180}+\sum_{r=1}^{t-1} \frac{8n^3+n(2p^{4r}+20p^{2r})}{180p^{2r}}+\frac{8n^3+n(2p^{4t}+20p^{2t})}{360p^{2t}}.$$
\end{proof}

If $n={p_{1}}^{a_{1}}{p_{2}}^{a_{2}}\cdots{p_{m}}^{a_{m}}$ with $p_{i}$'s be distinct odd primes and  $m \geq 2$, following similar arguments used in previous results, one can prove;
\begin{theorem}\label{thm3.6}
Let $n={p_{1}}^{a_{1}}{p_{2}}^{a_{2}}\cdots{p_{m}}^{a_{m}}$ with $p_{i}$'s distinct odd primes, $a_i$'s are postive integers and  $m \geq 2$. Then $h(d) \geq 2(a_{1}+a_{2}+\cdots+ a_{m})-m+2$. And $h(d) = 2(a_{1}+a_{2}+\cdots+ a_{m})-m+2$ if and only if
\begin{eqnarray*}
\zeta_{k_n}(-1)&=&\frac{n^3+5n}{36}+\sum_{\substack{1 \leq i \leq m\\ 1 \leq r_{i} \leq a_{i}-1}} \frac{8n^3+n(2{p_{i}}^{4r{i}}+20{p_{i}}^{2r_{i}})}{180{p_{i}}^{2r_{i}}}\\
 &+& \sum_{i=1}^{m} \frac{8n^3+n(2{p_{i}}^{4a{i}}+20{p_{i}}^{2a_{i}})}{360{p_{i}}^{2a_{i}}}.
\end{eqnarray*}
\end{theorem}

\noindent{\textit{Remark:}} If all the $a_i$'s and $t$ are zero, then $n=1$ and $d=2$. Hence $h(2)=1$.

\section{\textbf{The field $\mathbf{\mathbb{Q}(\sqrt{n^2+4})}$}}
We now study the class number of $k_n=\mathbb{Q}(\sqrt{d})$, where $d=n^2+4$ is a square-free positive integer. Clearly $d \equiv 5 \pmod 8$ and $n$ is odd. In this case, $\varepsilon= \frac{n+\sqrt{n^2+4}}{2}, \text { and } N(\varepsilon)=-4$.
Let $p|n$ then 
\begin{equation}\label{eqn4.1}
(p)=\left(p,\frac{p+2+\sqrt{d}}{2}\right)\left(p,\frac{p+2-
\sqrt{d}}{2}\right).
\end{equation}

By \cite[Theorem 2.3]{BK1}, we also know that
\begin{equation}\label{eqn4.2}
\zeta_{k_n}(-1, \mathcal{P})=\frac{n^3+11n}{360}.
\end{equation}

\begin{lemma}\label{lem4.1}
Let $n^2+4 \equiv 5 \pmod 4$, $p^t \mid\mid n$ be a prime, $\mathfrak{a}=\left(p,\frac{p+2+\sqrt{d}}{2}\right)$ and $\mathfrak{a}^{-1}=\left(p,\frac{p+2-\sqrt{d}}{2}\right)$. Then $\{p^r,\frac{p+2+\sqrt{d}}{2}\}$ and $\{p^r,\frac{p+2-\sqrt{d}}{2}\}$ is an 
integral basis for $\mathfrak{a}^r$ and $\mathfrak{a}^{-r}$, respectively, $ \forall \hspace*{2mm} 1 \leq r \leq t$. 
\end{lemma}

\begin{proof}
Consider $$M_{r}=\Big[p^r,\frac{p^r+2+\sqrt{d}}{2}\Big],$$ 
a nonzero $\mathbb{Z}$-module in $\mathcal{O}_{k_n}$.
Then, by \cite[Proposition 2.6 and 2.11]{Fra}, $M_{r}$ is an ideal and and $N(M_{r})=p^r$. As $ \mathfrak{a}^r \subseteq M_{r}$, $\forall \hspace{2mm} 1 \leq r \leq t$, and $N(\mathfrak{a}^r)=p^r$, one has $ M_{r} = \mathfrak{a}^r$ and hence $\{p^r,\frac{p+2+\sqrt{d}}{2}\}$ is an integral basis for $\mathfrak{a}^r$, $\forall \hspace{2mm} 1 \leq r \leq t$. Similarly $\{p^r,\frac{p+2-\sqrt{d}}{2}\}$  is an integral basis for $\mathfrak{a}^{-r},$ $\forall \hspace{2mm} 1 \leq r \leq t$.
\end{proof}

\begin{theorem}\label{thm4.1}
Let $d=n^2+4$ and  $n=p^t$ with $p$ be an odd prime and $t \geq 1$ be an integer. Then $h(d) \geq t$ and equality holds if and only if  $$\zeta_{k_n}(-1)=\frac{n^3+11n}{360}+\sum_{r=1}^{t-1} \frac{n^3+n(p^{4r}+10p^{2r})}{360p^{2r}}.$$
\end{theorem}
\begin{proof}
Let $\mathcal{A}$ be an ideal class containing $\mathfrak{a}=\left(p,\frac{p+2+\sqrt{d}}{2}\right).$
Then by using Lemma \ref{lem2.1}, Lemma \ref{lem2.2}, Lemma \ref{lem4.1} and Theorem \ref{thm2.1} we obtain: 
$$\zeta_{k_n}(-1, \mathcal{A}^{r})=\frac{n^3+n(p^{4r}+10p^{2r})}{360p^{2r}}=\zeta_{k_n}(-1, \mathcal{A}^{-r}), \hspace{5mm} \forall \hspace*{2mm} 1\leq r \leq t.$$
 If for any $1\leq r \leq t$, $\zeta_{k_n}(-1, \mathcal{P})=\zeta_{k_n}(-1, \mathcal{A}^r)$, then $n=p^r$. Hence $|\mathcal{A}| \geq t$. This implies that $h(d) \geq t$ and $h(d)=t$ if and only if $$\zeta_{k_n}(-1)=\zeta_{k_n}(-1, \mathcal{P})+\sum_{r=1}^{t-1}\zeta_{k_n}(-1, \mathcal{A}^r).$$
\end{proof}

Using similar arguments one gets:

\begin{theorem}\label{thm4.2}
Let $n={p_{1}}^{a_{1}}{p_{2}}^{a_{2}}\cdots{p_{m}}^{a_{m}}$ with $p_{i}$'s distinct odd primes and $a_i$'s are postive integers.
\begin{itemize}
\item[(i)] If $m>2$, then $h(d) \geq 2(a_{1}+a_{2}+\cdots+ a_{m})-m+1$.
Equality holds (only if $m$ is odd) if and only if 
\begin{eqnarray*}
\zeta_{k_n}(-1)&=&\frac{n^3+11n}{360}+\sum_{\substack {1 \leq i \leq m\\ 1 \leq r_{i} \leq a_{i}-1}} \frac{n^3+n({p_{i}}^{4r_{i}}+10{p_{i}}^{2r_{i}})}{180{p_{i}}^{2r_{i}}}\\
& +& \sum_{i=1}^{m} \frac{n^3+n({p_{i}}^{4a_{i}}+10{p_{i}}^{2a_{i}})}{360{p_{i}}^{2a_{i}}}.
\end{eqnarray*}
\item[(ii)] If $m=2$, then $h(d) \geq 2(a_{1}+a_{2})-2$ and equality holds if and only if  
\begin{eqnarray*}
\zeta_{k_n}(-1)&=&\frac{n^3+11n}{360}+\sum_{\substack {1 \leq i \leq 2\\ 1 \leq r_{i} \leq a_{i}-1}} \frac{n^3+n({p_{i}}^{4r_{i}}+10{p_{i}}^{2r_{i}})}{180{p_{i}}^{2r_{i}}}\\
& +& \frac{n^3+n({p_1}^{4a_1}+10{p_1}^{2a_1})}{360{p_1}^{2a_1}}.
\end{eqnarray*}
\end{itemize}
\end{theorem}

\noindent{\textit{Remark:}} If all the $a_i$'s and $t$ are zero, then $n=1$ and $d=5$. Hence $h(5)=1$.

\section{\textbf{Remarks on Chowla and Yokoi's conjecture}}
For the sake of completeness, we recall the conjectures again.

\begin{theorem}\label{thm5.1}
If $d=m^2+1$ is a prime with $m>26$, then 
$h(d)>1$.
\end{theorem}

\begin{theorem}\label{thm5.2}
 Let $d=m^2+4$ be a square-free integer for 
some positive integer $m$. Then there exist exactly $6$ real 
quadratic fields $\mathbb{Q}(\sqrt{d})$ of class number one, 
viz.
$m\in\{1,3,5,7,13,17\}$.
\end{theorem}

\begin{proof}[Proof of Theorem 1.1]
Assume $d=m^2+1$ to be just square-free in Theorem \eqref{thm5.1}. By Theorem \ref{thm3.1}-\ref{thm3.6}, we deduce that, $h(d)=1$ $\Rightarrow$ $n=1 \text { or } 2p$, for some prime $p$. Hence, if $m\neq 1,2p$, where $p$ is a prime, then $h(m^2+1)>1$.\\
\end{proof}

 Thus one needs only to consider the family $k_p=\mathbb{Q}(\sqrt{4p^2+1})$ in Chowla's conjecture. But this family has been already covered by Bir\'{o} in \cite{BI03}. Hence we have the following theorem;

\begin{theorem}
If $d=m^2+1$ is a square-free integer with $m>26$, then 
$h(d)>1$.
\end{theorem}

 Now by Proposition \ref{prop2.1},
\begin{eqnarray*}
\zeta_{k_p}(-1)&=&\frac{1}{60}\sum_{\substack{ |t|<\sqrt{4p^2+1}\\ t^2\equiv 4p^2+1\pmod 4}}\sigma\left(\frac{4p^2+1-t^2}{4}
\right)\\
&=&\frac{1}{60}\sum_{\substack{ |t|\leq {2p}\\ t \text { 
\textit{is} } odd}}\sigma\left(\frac{4p^2+1-t^2}{4}\right)\\
&=&\frac{2}{60}\sum_{0 \leq n \leq {p-1}} \sigma
\left(\frac{4p^2+1-(2n+1)^2}{4}\right)\\
&=&\frac{1}{30}\sum_{0 \leq n \leq {p-1}} \sigma(p^2-n(n+1))\\
&\geq &\frac{1}{30}\sum_{0 \leq n \leq {p-1}}\{1+p^2-n(n+1)\}\\
&=& \frac{8p^3+28p}{360}.
\end{eqnarray*}
By \cite[Theorem 2.4]{BK1}, we have $h(\sqrt{4p^2+1})=1$ $\Longleftrightarrow$ $\zeta_{k_p}=\frac{8p^3+28p}{360}$. In the above inequality, 
equality occurs only if $\{p^2-n(n+1)\}_{n=1}^{p-2}$ 
consists of only prime numbers.
Thus, one can have an alternate proof of Chowla's conjecture if the following lemma can be proved:

\begin{lemma}\label{lem5.1}
The set $\{p^2-n(n+1)\}_{n=1}^{p-2}$ always contains a composite number, except for $p = 2,3,5,7 \text { and } 13$.
\end{lemma}

That is, Chowla's conjecture is equivalent to above lemma and  hence Theorem \eqref{thm1.3} is proved.

\begin{proof}[Proof of Theorem 1.4]
Consider $m^2+4$ to be a square-free integer, then
by Theorem \ref{thm4.1}-\ref{thm4.2}, $h(d)=1$ $\Rightarrow$ $m=1,p$, for some prime $p$.
\end{proof}

 So we only have to consider the family $k_p=
\mathbb{Q}(\sqrt{p^2+4})$ to prove Yokoi's conjecture. Again by 
Proposition \ref{prop2.1}, we have
\begin{eqnarray*}
\zeta_{k_p}(-1)&=&\frac{1}{60}\sum_{\substack{ |t|<\sqrt{p^2+4}\\ t^2\equiv p^2+4\pmod 4}}\sigma\left(\frac{p^2+4-t^2}{4}\right)
\\
&=&\frac{1}{60}\sum_{\substack{ |t|\leq {p}\\ t \text { 
\textit{is} } odd}}\sigma\left(1+\frac{p^2-t^2}{4}\right)\\
&=&\frac{2}{60}\sum_{\substack{ t\leq {p}\\ t \text { 
\textit{is} } odd}}\sigma\left(1+\frac{(p-t)(p+t)}{4}\right).
\end{eqnarray*}

If we replace $t$ by $p-k$, where k is even, then

\begin{equation*}
\zeta_{k_p}(-1)=\frac{1}{30}\sum_{\substack{ 0 \leq k\leq {p-1}\\ k \text { \textit{is} } even}}\sigma\left(1+\frac{(k)(2p-k)}
{4}\right).
\end{equation*}
 Now replace $k$ by $2n$, then we have $0 \leq n \leq \frac{p-1}
 {2}$. Thus
 
\begin{eqnarray*}
\zeta_{k_p}(-1)&=&\frac{1}{30}\sum_{ 0 \leq n \leq {(p-1)/2}}
\sigma\left(1+\frac{(2n)(2p-2n)}{4}\right)\\
&=&\frac{1}{30}\sum_{ 0 \leq n \leq {(p-1)/2}}\sigma(1+pn-n^2)\\
&\geq &\frac{1}{30}\sum_{0 \leq n \leq {(p-1)/2}}\{1+pn-n^2\}\\
&=& \frac{p^3+11p}{360}.
\end{eqnarray*}

Now by \cite[Theorem 2.4]{BK1}, we have $h(\sqrt{p^2+4})=1$ $\Longleftrightarrow$ $\zeta_{k_p}=\frac{p^3+11p}{360}$. From above, equality occurs only if the set $\{1+pn-n^2\}_{n=1}
^{\frac{p-1}{2}}$ consist of only prime numbers. Since Yokoi's conjecture has been proved by Bir\'{o} in \cite{B03}, therefore Lemma \eqref{lem5.2} is true. But if one can prove this lemma by some other way then we will have a different prove of Yokoi's conjecture.

\begin{lemma}\label{lem5.2}
The set $\{1+pm-m^2\}_{m=1}^{\frac{p-1}{2}}$ always contains a composite number, except for $p = 3,5,7,13 \text { and } 17$.
\end{lemma}

Hence Theorem \eqref{thm1.5} is proved.\\

Now let $g(n)$ be the least prime number which is a quadratic residue modulo  
$n$. Chowla and Friedlander \cite{CF76} proved that if $p=m^2+1$  is a prime, 
$m>2$ and $h(p)=1$, then  $g(p)=\frac{m}{2}$. Using Lemma 
\eqref{lem5.1} we will get some upper bound for $g(4p^2+1)$, if $h(4p^2+1)>1$, irrespective of $(4p^2+1)$ is a prime or not.

\begin{theorem}\label{thm5.3}
Let $d=4p^2+1$ be a square-free integer, where $p$ is a prime. 
Consider $\mathbb{Q}(\sqrt{d})$ then $g(d)<p$, 
except for $p=2,3,5,7 \text { and } 13$.
\end{theorem}

\begin{proof}
By Lemma \eqref{lem5.1}, $\forall \hspace*{2mm} p$, except $p = 
2,3,5,7 \text { and } 13$, there exist a $1 \leq n_0 \leq p-2$ 
such that $p^2-n_0(n_0 +1)$ is not a prime number. Since $p^2-
n_0(n_0 +1)<p^2$, therefore there exist $q<p$ such that $q|(p^2-
n_0(n_0 +1))$. So we have 
$$p^2-n_0(n_0 +1) \equiv 0 \pmod q.$$
This implies $p^2-n(n +1)$ has a solution $n_0$ in $
\mathbb{F}_q$ and $$n_0=\frac{-1\pm \sqrt{1+4p^2}}{2}.$$
Therefore, $1+4p^2$ is 
a square in $\mathbb{F}_q$, i.e. $1+4p^2=x^2$ in $\mathbb{F}_q
$, for some $x \in \mathbb{F}_q$ . Thus $1+4p^2$ is a quadratic residue of some prime 
$q<p$. Hence $q$ is also a quadratic 
residue of $1+4p^2$, since $1+4p^2 \equiv 1 \pmod 4$.
\end{proof}

Similarly, using Lemma \eqref{lem5.2} one can deduce:
\begin{theorem}\label{thm5.4}
Let $d=p^2+4$ be a square-free integer, where $p$ is a prime. 
Consider $\mathbb{Q}(\sqrt{d})$ then $g(d)<p$, 
except for $p=3,5,7,13 \text { and } 17$.
\end{theorem}

\section{\textbf{Class group of prime power order}}
In this section, we deduce some conditions on the 
exponents of prime factors of $n$ so that the class group of  prime power order is cyclic.  

\begin{theorem}\label{thm6.1}
Let $n=2p^t$ with $p$ an odd prime, $t \geq 1$ an integer, 
$d=n^2+1\equiv 5 \pmod 8$. Let $h(d)=q^r$ for some prime $q$ and positive 
integer $r\geq 2.$ If $t>q^{r-1}$, then $\mathfrak{C}(k_n)\cong \mathbb{Z}_4$.
\end{theorem}

\begin{proof}
By theorem \eqref{thm3.1}, we have $h(d)\geq t$ and if $\mathcal{A}$ is an ideal class containing $\mathfrak{a}=
\left(p,\frac{1+\sqrt{d}}{2}\right)$, then $\mathcal{A}$ is non-
principal, $|\mathcal{A}|>q^{r-1}$ and $|\mathcal{A}|\big|q^r$. 
This implies $|\mathcal{A}|=q^r$ and hence class group is 
cyclic.
\end{proof}

The following results can also be proved using theorems
in $\S3$ and $\S4$ and some group theoretic arguments.
\begin{theorem}\label{thm6.2}
Let $d=n^2+1\equiv 5 \pmod 8$. If 
$h(d)=q^r$ for some prime $q$ and positive integer $r\geq 2$, 
then,\\
\textbf{I.} Let $n=2{p_{1}}^{a_{1}}{p_{2}}^{a_{2}}\cdots{p_{m}}
^{a_{m}}$ with $p_{i}$'s be distinct odd primes and  $m>2$. If 
one of the following holds:
\begin{itemize}
\item[(i)] $a_{i}\geq \frac{q^{r-1}+1}{2}$ for any $1 \leq i 
\leq m$.
\item[(ii)] $a_i \geq \frac{q^{r-j}+1}{2}$ and $a_l\geq  
\frac{q^{j-1}+1}{2}$ for some $i \neq l$ and $2 \leq j <r$.
\end{itemize} 
Then $\mathfrak{C}(k_n)\cong \mathbb{Z}_{q^r}$.\\
\textbf{II.}  Let $n=2{p_{1}}^{a_{1}}{p_{2}}^{a_{2}}$ with 
$p_{1},p_{2}$ be distinct odd primes. If $a_{1} \text{ or } 
a_{2} \geq \frac{q^{r-1}+1}{2}$, then $\mathfrak{C}(k_n)\cong \mathbb{Z}_{q^r}$.
\end{theorem}

\begin{theorem}\label{thm6.3}
Let $d=n^2+1\equiv 1 \pmod 8$. If 
$h(d)=q^r$ for some prime $q$ and positive integer $r\geq 2$, 
then,\\
\textbf{I.} Let $n=2^sp^t$ with $p$ be an odd prime, $s \geq 2$ 
and $t \geq 1$ be an integers. If $t \geq \frac{q^{r-1}+1}{2}$ 
or $s-1 \geq \frac{q^{r-1}+1}{2}$.
%\item[(ii)] $t \geq \frac{q^{r-j}+1}{2}$ and $s-1 \geq 
%\frac{q^{j-1}+1}{2}$ (or vive-versa) for any $2 \leq j <r$.
%\item[(iii)] If $r=2$ and $t or s-1 \geq \frac{q-1}{2}$. 
Then $\mathfrak{C}(k_n)\cong \mathbb{Z}_{q^r}$.\\
%\textbf{II.} If $n=2^s{p_{1}}^{a_{1}}{p_{2}}^{a_{2}}$ with 
%$p_{1},p_{2}$ be distinct odd primes.
%\begin{itemize}
%\item[(i)]  $a_{1},a_{2} \text{ or } s-1 \geq \frac{q^{r-1}+1}
%{2}$.
%\item[(ii)] If $r>2$ , $a_i \geq \frac{q^{r-j}+1}{2}$ and $s-1 
%\geq \frac{q^{j-1}+1}{2}$ for any $i$(or vive-versa) and for 
%any $2\leq j <r$.
%\item[(iii)] If $r=2$ and $a_i or s-1 \geq \frac{q-1}{2}$ for 
%$i=1,2.$
%\end{itemize} 
%Then class group is cyclic.\\
\textbf{II.} Let $n=2^s{p_{1}}^{a_{1}}{p_{2}}^{a_{2}}
\cdots{p_{m}}^{a_{m}}$ with $p_{i}$'s be distinct odd primes, $s 
\geq 2$ and  $m\geq2$. If one of the following holds:
\begin{itemize}
\item[(i)] $a_{i} \text{ or } s-1\geq \frac{q^{r-1}+1}{2}$ for 
any $1 \leq i \leq m$.
\item[(ii)] $a_i \geq \frac{q^{r-j}+1}{2}$ and $a_l\text{ or } 
s-1 \geq  \frac{q^{j-1}+1}{2}$ for some $i \neq l$ and $2 \leq j 
<r$.
%\item[(iii)] If $r=2$ and $a_i od s-1 \geq \frac{q-1}{2}$ for 
%some $i.$
\end{itemize} 
Then $\mathfrak{C}(k_n)\cong \mathbb{Z}_{q^r}$.
\end{theorem}

\begin{theorem}\label{thm6.4}
Let $d=n^2+1 \equiv 2 \pmod 4$. If 
$h(d)=q^r$ for some prime $q$ and positive integer $r\geq 2$, 
then,\\
\textbf{I.} Let $n=p^t$ with $p$ be an odd prime and $t \geq 1$ 
an integer. If $t  \geq \frac{q^{r-1}+1}{2}$, then $\mathfrak{C}(k_n)\cong \mathbb{Z}_{q^r}$.\\
\textbf{II.} If $n={p_{1}}^{a_{1}}{p_{2}}^{a_{2}}\cdots{p_{m}}
^{a_{m}}$ with $p_{i}$'s be distinct odd primes and  $m \geq 2$. 
If one of the following holds:
\begin{itemize}
\item[(i)] $a_{i} \geq \frac{q^{r-1}+1}{2}$ for any $1 \leq i 
\leq m$.
\item[(ii)] $a_i \geq \frac{q^{r-j}+1}{2}$ and $a_l \geq 
\frac{q^{j-1}+1}{2}$  for some $i \neq l$ and for any $2 \leq j 
<r$.
\end{itemize} 
Then $\mathfrak{C}(k_n)\cong \mathbb{Z}_{q^r}$.
\end{theorem}

\begin{theorem}\label{thm6.5}
Let $d=n^2+4 \equiv 5 \pmod 8$. If $h(d)=q^r$ 
for some prime $q$ and positive integer $r\geq 2$, then,\\
\textbf{I.} Let $n=p^t$ with $p$ be an pdd prime and $t \geq 1$ 
an integer. If $t  \geq q^{r-1}+1$, then $\mathfrak{C}(k_n)\cong \mathbb{Z}_{q^r}$.\\
\textbf{II.} Let $n={p_{1}}^{a_{1}}{p_{2}}^{a_{2}}$ with 
$p_{1},p_{2}$ be distinct odd primes. If $a_{1} \text{ or }a_{2} 
\geq \frac{q^{r-1}+1}{2}$, then $\mathfrak{C}(k_n)\cong \mathbb{Z}_{q^r}$.\\
\textbf{III.} Let $n={p_{1}}^{a_{1}}{p_{2}}^{a_{2}}\cdots{p_{m}}
^{a_{m}}$ with $p_{i}$'s be distinct odd primes and  $m>2$. If 
one of the following holds:
\begin{itemize}
\item[(i)] $a_{i} \geq \frac{q^{r-1}+1}{2}$ for any $1 \leq i 
\leq m$.
\item[(ii)] $a_i \geq \frac{q^{r-j}+1}{2}$ and $a_l \geq 
\frac{q^{j-1}+1}{2}$  for some $i \neq l$ and for any $2 \leq j 
<r$.
%\item[(iii)] If $r=2$ and $a_i \geq \frac{q-1}{2}$ for some $i.
%$
\end{itemize} 
Then $\mathfrak{C}(k_n)\cong \mathbb{Z}_{q^r}$.
\end{theorem}

%\begin{theorem}
%Let $n=2^sp^t$ with $s \geq 2,t \geq 1$ an integer. Let 
%$h(d)=q^r$ for some prime $q$ and positive integer $r.$
%\begin{itemize}
%\item[(i)] If either of $t \mbox{or} s-1 \geq \frac{q^{r-1}+1}
%{2}$.
%\item[(ii)] If $r>2$ , $t \geq \frac{q^{r-1}-1}{2}$ and $s-1 
%\geq \frac{q^2-1}{2}$ (or vive-versa).
%\item[(iii)] If $r=2$ and $t or s-1 \geq \frac{q-1}{2}$.
%\end{itemize} 
%Then class group is cyclic.
%\end{theorem}
%
%\begin{theorem}
%Let $n=2^s{p_{1}}^{a_{1}}{p_{2}}^{a_{2}}$ with $p_{1},p_{2}$ be 
%distinct odd primes. Let $a_{1}$ or $a_{2}>1$. Let $h(d)=q^r$ 
%for some prime $q$ and positive integer $r.$ If any of the 
%following holds: 
%\begin{itemize}
%\item[(i)] If either of $a_{1},a_{2} os s-1 \geq 
%\frac{q^{r-1}+1}{2}$.
%\item[(ii)] If $r>2$ , $a_i \geq \frac{q^{r-1}-1}{2}$ and $s-1 
%\geq \frac{q^2-1}{2}$ for any $i$(or vive-versa).
%\item[(iii)] If $r=2$ and $a_i or s-1 \geq \frac{q-1}{2}$ for 
%$i=1,2.$
%\end{itemize} 
%Then class group is cyclic.
%\end{theorem}

\section{\textbf{Corollaries}}
We shall now deduce some interesting corollaries to our main theorems.\\
\\
\textit{Case I:} $d=n^2+1\equiv 5 \pmod 8$.

\begin{corollary}\label{cor7.1}
Let $n=2p^t$ with $p$ be an odd prime and $t > 1$ an integer. If 
$t \leq h(d) \leq 2t-1$ then the class group is cyclic. 
\end{corollary}

\begin{corollary}\label{cor7.2}
Let $n=2{p_{1}}^{a_{1}}{p_{2}}^{a_{2}}\cdots{p_{m}}^{a_{m}}$ 
with $p_{i}$'s distinct odd primes, $a_i$'s be postive integer 
and  $m>2$. Then the following holds:
\begin{itemize}
\item[(i)] If $m$ is even then $h(d) \geq 2(a_{1}+a_{2}+\cdots+ 
a_{m})-m+2$
\item[(ii)]If $h(d)$ is odd then $h(d) \geq 2(a_{1}+a_{2}+\cdots
+ a_{m})+1$ and equality holds if and only if 
$$\zeta_{k_n}(-1)=\frac{n^3+14n}{360}+\sum_{\substack{1 \leq i 
\leq m\\ 1 \leq r_{i} \leq a_{i}}} \frac{n^3+n(4{p_{i}}^{4r_{i}}
+10{p_{i}}^{2r_{i}})}{180{p_{i}}^{2r_{i}}}.$$
\end{itemize}

\end{corollary}

\begin{corollary}\label{cor7.3}
Let $n=2{p_{1}}^{a_{1}}{p_{2}}^{a_{2}}$ with $p_{1},p_{2}$ be 
distinct odd primes. Then the following holds:
\begin{itemize}
\item[(i)] If $a_1 \text{ or } a_2>1$ and 
$h(d)$ is odd, then $h(d) \geq 2(a_{1}+a_{2})+1$ and equality 
holds if and only if 
$$\zeta_{k_n}(-1)=\frac{n^3+14n}{360}+\sum_{\substack{i=1,2 \\ 1 
\leq r_{i} \leq a_{i}}} \frac{n^3+n(4{p_{i}}^{4r_{i}}+10{p_{i}}
^{2r_{i}})}{180{p_{i}}^{2r_{i}}}.$$
\item[(ii)] If $a_1=a_2=1$, then $h(d)=3$ if and only if
$$\zeta_{k_n}(-1)=\frac{n^3+14n}{360}+ \frac{n^3+n(4{p_1}
^{4}+10{p_{1}}^{2})}{180{p_{1}}^{2}}.$$

\end{itemize}
\end{corollary}

\textit{Case II:} $d=n^2+1\equiv 1 \pmod 8$.

\begin{corollary}\label{cor7.4}
Let $n=2^sp^t$ with $s,t>1$ and $p$ be an odd prime. If $h(d)$ 
is odd, then $h(d) \geq 2(t+s)-1$ and equality holds if and only 
if 
\begin{eqnarray*}
\zeta_{k_n}(-1) &=& \frac{n^3+14n}{360}+\sum_{r=1}^{t} 
\frac{n^3+n(4p^{4r}+10p^{2r})}{180p^{2r}}\\
&+&\sum_{j=1}^{s-1} \frac{n^3+n(4\times2^{4j}+10\times2^{2j})}
{180\times2^{2j}}.
\end{eqnarray*}
\end{corollary}

\begin{corollary}\label{cor7.5}
Let $n=2^s{p_{1}}^{a_{1}}{p_{2}}^{a_{2}}\cdots{p_{m}}^{a_{m}}$ 
with $p_{i}$'s distinct odd primes and  $s,m \geq 2$. If $h(d)$ 
is odd, then $h(d) \geq 2(a_{1}+a_{2}+\cdots+ a_{m})+2s-1$ and 
equality holds if and only if  
\begin{eqnarray*}
\zeta_{k_n}(-1)&=&\frac{n^3+14n}{360}+\sum_{\substack{1 \leq i 
\leq m\\ 1 \leq r_{i} \leq a_{i}}} \frac{n^3+n(4{p_{i}}^{4r_{i}}
+10{p_{i}}^{2r_{i}})}{180{p_{i}}^{2r_{i}}}\\
&+&\sum_{j=1}^{s-1} \frac{n^3+n(4\times2^{4j}+10\times2^{2j})}
{180\times2^{2j}}.
\end{eqnarray*}
\end{corollary}

\textit{Case III:} $d=n^2+1\equiv 2 \pmod 4$.

\begin{corollary}\label{cor7.6}
Let $n=p^t$ with $t \geq 1$ an integer. If $2t \leq h(d) \leq 
4t-2$, then the class group is cyclic. 
\end{corollary}

\textit{Case IV:} $d=n^2+4\equiv 5 \pmod 8$.

\begin{corollary}\label{cor7.7}
Let $n=p^t$ with $t > 1$ an integer. If $t \leq h(d) \leq 2t-1$, 
then the class group is cyclic. 
\end{corollary}

\begin{corollary}\label{cor7.8}
Let $n={p_{1}}^{a_{1}}{p_{2}}^{a_{2}}\cdots{p_{m}}^{a_{m}}$ with 
$p_{i}$'s distinct odd primes and  $m>2$. If $m$ is even, then 
$h(d) > 2(a_{1}+a_{2}+\cdots+ a_{m})-m+2$ and if $h(d)$ is odd, 
then $h(d)=2(a_{1}+a_{2}+\cdots+ a_{m})+1$ if and only if 
$$\zeta_{k_n}(-1)=\frac{n^3+11n}{360}+\sum_{\substack {1 \leq i 
\leq m\\ 1 \leq r_{i} \leq a_{i}}} \frac{n^3+n({p_{i}}^{4r_{i}}
+10{p_{i}}^{2r_{i}})}{180{p_{i}}^{2r_{i}}}.$$
\end{corollary}

\begin{corollary}\label{cor7.9}
Let $n={p_{1}}^{a_{1}}{p_{2}}^{a_{2}}$ with $p_{1},p_{2}$ be 
distinct odd primes. Then the following holds:
\begin{itemize}
\item[(i)] If $a_1 \text{ or } a_2>1$ and 
$h(d)$ is odd, then $h(d) \geq 2(a_{1}+a_{2})+1$. Equality 
holds if and only if 
$$\zeta_{k_n}(-1)=\frac{n^3+11n}{360}+\sum_{\substack {1 \leq i 
\leq 2\\ 1 \leq r_{i} \leq a_{i}}} \frac{n^3+n({p_{i}}^{4r_{i}}
+10{p_{i}}^{2r_{i}})}{180{p_{i}}^{2r_{i}}}.$$

\item[(ii)] If $a_1=a_2=1$ then $h(d)=3$ if and only if
$$\zeta_{k_n}(-1)=\frac{n^3+11n}{360}+ \frac{n^3+n({p_{i}}
^{4}+10{p_{i}}^2)}{180p_{i}},$$
for any $i$.

\end{itemize}
\end{corollary}

We give the proof of Corollary \eqref{cor7.1} and Corollary 
\eqref{cor7.2}(ii) and the other cases can be handled similarly.

\begin{proof}[Proof of Corollary 7.1]
If $\mathfrak{a}=\left(p,\frac{1+\sqrt{d}}{2}\right) \in 
\mathcal{A}$ then by Theorem \eqref{thm3.1}, $|\mathcal{A}|\geq 
t$. Since $|\mathcal{A}|\big|h(d)$, $|\mathcal{A}|\geq t$ and 
$h(d) \leq 2t-1$, therefore $|\mathcal{A}|=h(d)$. Hence the class 
group is cyclic.
\end{proof}

\begin{proof}[Proof of Corollary 7.2(ii)]
By Theorem \eqref{thm3.2}, if $\mathfrak{a}_i=\left(p_i,\frac{1+
\sqrt{d}}{2}\right) \in \mathcal{A}_i$, then $|\mathcal{A}_i|\geq 
2a_i$. If $h(d)$ is odd, then $|\mathcal{A}_i|\geq 
2a_i+1$ (as $|\mathcal{A}|\big|h(d)$). Hence  $h(d) 
\geq 2(a_{1}+a_{2}+\cdots+ a_{m})+1$ and equality holds if and 
only if 
$$\zeta_{k_n}(-1)=\zeta_{k_n}(-1, \mathcal{P})+\sum_{\substack{1 
\leq i \leq m\\ 1 \leq r_{i} \leq a_{i}}} 2\zeta_{k_n}(-1, 
\mathcal{A}_i^{r_i}).$$
Thus equality holds if and only if
$$\zeta_{k_n}(-1)=\frac{n^3+14n}{360}+\sum_{\substack{1 \leq i 
\leq m\\ 1 \leq r_{i} \leq a_{i}}} \frac{n^3+n(4{p_{i}}^{4r_{i}}
+10{p_{i}}^{2r_{i}})}{180{p_{i}}^{2r_{i}}}.$$
\end{proof}

\section{\textbf{Concluding remarks}}
We believe that our method should go through for other R-D type real quadratic fields also. It will be interesting to extend these reults for other real 
quadratic fields, whose fundamental unit is known. Then one can try to reduce 
the class number 1 problem for that particular family to its subfamily, like 
we did for Chowla and Yokoi's conjecture. One can also say more 
clearly about the prime power order class group. In $\S6$, 
if $r$ is small, say $2$ or $3$, then in most cases we can 
exactly determine the class group by just looking at the exponents 
of prime factors.

\section*{\textbf{Acknowledgements}} 
\noindent I am  grateful to my supervisor Prof. K. Chakraborty for his constant support, comments and suggestions while doing this project. I would like to express my gratitude to Dr. Azizul Hoque for many fruitful discussions. I would also like to thank my friend and colleague Mr. Rishabh Agnihotri for his valuable suggestions and discussions. This work is partially supported by Infosys grant.
%%%%%%%%%%%%%%%%%%%%%%%%%%%%%%%%%%%%%%%%%%%%%%%%%%%      


\begin{thebibliography}{10}


\bibitem{AP50} T. M. Apostol, {\it Generalized Dedekind sums and 
transformation formulae of certain Lambert series}, Duke Math. 
J., \textbf{17} (1950), 147--157.

\bibitem{B03} A. Bir\'{o}, {\it Yokoi's conjecture}, Acta 
Arith., {\bf 106} (2003), 85--104.

\bibitem{BI03} A. Bir\'{o}, {\it Chowla's conjecture}, Acta 
Arith., {\bf 107} (2003), 179--194.

\bibitem{BK1} D. Byeon and H. K. Kim, {\it Class number $1$ 
criteria for real quadratic fields of Richaud-Degert type}, J. 
Number Theory, \textbf{57} (1996), 328--339.

\bibitem{BK2} D. Byeon and H. K. Kim, {\it Class number $2$ 
criteria for real quadratic fields of Richaud-Degert type}, J. 
Number Theory,  \textbf{62} (1998), 257--272.

\bibitem{CH-19} K. Chakraborty and A. Hoque, {\it On the plus parts of the class numbers of cyclotomic fields}, Preprint.

\bibitem{CHM18} K. Chakraborty, A. Hoque and M. Mishra, {\it A 
note on certain real quadratic fields with class number up to 
three},  Kyushu J. Math.(To appear). \url{arXiv:1812.02488}
 
\bibitem{CHM19} K. Chakraborty, A. Hoque and M. Mishra, {\it A 
Classification of order 4 class groups of $\mathbb{Q}
(\sqrt{n^2+1})$}, submitted for publication. \url{arXiv:
1902.05250 } 


\bibitem{CF76} S. Chowla and J. Friedlander, {\it Class numbers 
and quadratic residues}, Glasgow Math. J., {\bf 17} (1976), no. 
1, 47--52.

\bibitem{Has65}  H. Hasse, {\it {\"U}ber mehrklasiige, aber 
eingeschlectige reell-quadratische Zahlk{\"o}rper}, Elemente der 
Mathematik, {\bf 20} (1965), 49--59. 


\bibitem{HS15} A. Hoque and H. K. Saikia, {\it On the class-number of the maximal real subfield of a cyclotomic field}, Quaest. Math., {\bf 39} (2016), no. 7, 889--894.

\bibitem{CH-18} A. Hoque and K. Chakraborty, {\it Pell-type equations and class number of the maximal real subfield of a cyclotomic field}, Ramanujan J., {\bf 46} (2018), no. 3, 727--742.

\bibitem{KLO87} H. K. Kim, M. G. Leu, and T. Ono, {\it On two 
conjectures on real quadratic fields}, Proc. Japan Acad. Ser. A 
Math. Sci., {\bf 63} (1987), no. 6, 222--224.

\bibitem{LAN} H. Lang, {\it \"{U}ber eine Gattung elemetar-
arithmetischer Klassen invarianten reell-quadratischer Zhalk
\"{o}rper},  J. Reine Angew. Math., {\bf 233} (1968), 123--175. 

\bibitem{Fra} F. Lemmermeyer, {\it Lecture notes on Algebraic Number Theory}, Bilkent University, 2007.  \url{http://www.fen.bilkent.edu.tr/~franz/ant06/ant.pdf}

\bibitem{Mol86} R. A. Mollin, { \it Lower bounds for class 
numbers of real quadratic fields}, Proc. Amer. Math. Soc., {\bf 
96} (1986), 545--550.

\bibitem{Mol87} R. A. Mollin, { \it Lower bounds for class 
numbers of real quadratic and biquadratic fields}, Proc. Amer. 
Math. Soc., {\bf 101} (1987), 439--444. 

\bibitem{Mol872} R. A. Mollin. { \it On the insolubility of a class of Diophantine equations and the nontriviality of the class numbers of related real quadratic fields of Richaud-Degert type}, Nagoya Math. J., {\bf 105} (1987), 39--47. 

\bibitem{MW88}  R. A. Mollin and H. C. Williams, {\it A 
conjecture of S. Chowla via the generalized Riemann hypothesis}, 
Proc. Amer. Math. Soc., {\bf 102} (1988), 794--796.  


\bibitem{SI69} C. L. Siegel, {\it Berechnung von Zetafunktionen 
an ganzzahligen Stellen}, Nachr. Akad. Wiss. G\"{o}ttingen 
Math.-Phys. Kl. II, {\bf 10} (1969), 87--102.

\bibitem{Yok68} H. Yokoi, {\it On real quadratic fields 
containing units with norm $-1$}, Nagoya Math. J. {\bf 33} 
(1968), 139--152. 
 
\bibitem{Yok70} H. Yokoi, {\it On the fundamental unit of real 
quadratic fields with norm $1$}, J. Number Theory, {\bf 2} 
(1970), 106--115. 
 
\bibitem{YO86} H. Yokoi, {\it Class-number one problem for 
certain kind of real quadratic fields}, Proc. International 
Conference on Class Numbers and Fundamental Units of Algebraic 
Number Fields (Katata, 1986), Nagoya Univ., Nagoya, 1986, 
125--137.


\end{thebibliography}
\end{document}